\documentclass[11pt,a4paper]{article}
\usepackage{amsmath, amsthm, amssymb, color}
\usepackage[top=2.0cm, bottom=2.0cm, left=2.0cm, right=2.0cm] {geometry}	
\usepackage{graphicx}
\usepackage{hyperref}
\newtheorem{dl}{Theorem}[section]

\newtheorem{bd}[dl]{Lemma}
\newtheorem{GT} [dl]{Assumption}

\title{On the existence of negative moments for some non-colliding particle systems and its application} 
\author{Minh-Thang Do\footnote{Email: thanghg99@gmail.com} \qquad Hoang-Long Ngo\footnote{Corresponding author. Email: ngolong@hnue.edu.vn}} 
\date{Hanoi National University of Education\\ 136 Xuan Thuy - Cau Giay - Hanoi - Vietnam}

\begin{document}
\maketitle 
	\begin{abstract}
We consider a class of $d$-dimensional stochastic differential equations that model a non-colliding random particle system. We provide a sufficient condition, which does not depend on the dimension $d$, for the existence of negative moments of the gap between two particles, and then apply this result to study the strong rate of convergence of the semi-implicit Euler-Maruyama approximation scheme. Our finding improves a recent result of Ngo and Taguchi (Annals of Applied Probability, 2020).    
\end{abstract}

 \textbf{MSC2010 subject classifications:} Primary 60K35; secondary 41A25.
 

\section{Introduction}
We consider a system  $X=(X_1, \ldots, X_d)$ defined by the following stochastic differential equations (SDEs),
	\begin{equation}\label{eqn1.1}
		\begin{cases} 
			& d X_{i}=\sum\limits_{j \neq i} \dfrac{\lambda}{X_{i}-X_{j}} d t+ b_i(X_i)dt+ \sigma_{i}\left(X_{i}\right) d B_{i}, \quad i=1, \ldots, d, \\
			& \left(X_{1}(0), \ldots, X_{d}(0)\right)=v \in \Delta_{d},
		\end{cases},
	\end{equation}
	where  $B= (B_1, \ldots, B_d)$ is a  $d$-dimensional Brownian motion defined on a filtered probability space  $(\Omega, \mathcal{F},(\mathcal{F}_t), \mathbb P)$ which satisfies the usual condition, $\lambda$ is a positive constant, $(b_i)_{1\leq i\leq d}$ and $(\sigma_i)_{1\leq i \leq d}$ are measurable functions defined on $\mathbb{R}$, and  
	$\Delta_{d}=\left\{x \in \mathbb{R}^{d}: x_{1}<x_{2}<\cdots<x_{d}\right\}.$

 The system $X$ can be used to describe positions of $d$ ordered particles involving in $\mathbb{R}$ in mathematical physics.  It also appears in many other fields such as in financial mathematics to describe the Cox-Ingerson-Ross short-term interest rate  (\cite{Bru90}),  or in random matrix theory to describe Dyson-Brown motions (see \cite{Dyson, AGZ10,   Ka15, KaTa04, KaTa11}). The existence and uniqueness of the solution to a class of general random non-colliding particle systems have been studied in    \cite{CeLe97, CeLe01, GrMa14}. The first paper to discuss the numerical approximation for SDE \eqref{eqn1.1} is \cite{LiMe}, where the author introduces a tamed-Euler approximation scheme. However, this scheme does not preserve the non-colliding property of the exact solution. Recently,   semi-implicit Euler-Maruyama and semi-implicit Milstein approximation schemes, which preserve the non-colliding property of the exact solution,  have been introduced in \cite{NT17} and \cite{LN}. It has been shown that under the condition $\frac{3\lambda}{d \|\sigma\|^{2}_{\infty}} \geq 1$,   the gap between two particles has  bounded negative moments of order $p$ for any $ p \in [0,  \frac{3\lambda}{d \|\sigma\|^{2}_{\infty}} -1]$ (see, Lemma 3.4 in \cite{NT17}). Using the existence of such negative moments, it can be shown that the semi-implicit Euler-Maruyama scheme and semi-implicit Milstein  scheme converge at the rate of order $1/2$ and $1$, respectively, in $L^2$-norm (see, Theorem 3.6 in \cite{NT17} and Theorem 1 in \cite{LN} for more details). 

Note that the condition $\frac{3\lambda}{d\|\sigma\|^{2}_{\infty}} \geq 1$ is quite restrictive since, in practice, the number of particles $d$ is usually  quite large. The main aim of this paper is to  relax such condition on $d$. In particular, by using a suitable change of probability measure, we show that if $\frac{\lambda}{\|\sigma\|_\infty^2}  > 2$ then  the gap  between two particles has bounded negative moment of order $p$ for any $p \in\big(0, \frac 16 \big(\frac {2\lambda}{\|\sigma\|^{2}_{\infty}}-1\big) \big)$. Then we apply this result to study the convergence of the semi-implicit schemes introduced in \cite{NT17} and \cite{LN}. 

	\section{Main results}
	\subsection{The existence and uniqueness of solution}
	In this section, we recall  a result  on the  existence and uniqueness of solution for non-colliding particle systems (see \cite{GrMa14} for more details). 
	We first introduce some assumptions on the coefficients of equation \eqref{eqn1.1}. 
	\begin{GT}
	\item[(B1)] The functions $b_{i} : \mathbb R \to \mathbb R$ are Lipschitz continuous for every $i=1,\ldots,d $.

	\item[(B2)] There exists a function $\rho: \mathbb{R}^{+} \rightarrow \mathbb{R}^{+}$ such that $\displaystyle\int\limits_{0^{+}} \rho^{-1}(|x-y|) d x=\infty$ and
			$$\left|\sigma_{i}(x)-\sigma_{i}(y)\right|^{2} \leq \rho(|x-y|), \quad x, y \in \mathbb{R}, i=1, \ldots, d.$$
	Moreover, the functions $ \sigma_{i}: \mathbb R \to \mathbb R$ are continuous for every $i=1,\ldots,d $.
    \item[(B3)]     $\|b\|_\infty := \sup_{x \in \mathbb R} \sup_{1 \leq i \leq d} |b_i(x)| < \infty.$
   
    \item[(B4)]  $\|\sigma\|^2_\infty := \sup_{x \in \mathbb R} \sup_{1 \leq i \leq d}  \sigma_{i} ^{2}(x)\leq 2\lambda.$
    
    \item[(B5)] If $i<j$ then $b_{i}(x) \leq b_{j}(x)$ for all $x \in \mathbb{R}$.

	\end{GT}
	The following result on the existence and uniqueness of solution to equation \eqref{eqn1.1}  is a consequence of Theorem 2.2 in \cite{GrMa14}:
\begin{dl}	\label{theorem2.2}
			
If the conditions (B1)-(B5) hold, then there exists a unique strong non-exploding solution to equation \eqref{eqn1.1}. The solution is non-colliding in finite time,  i.e.,  the first collision time
	$$\mathcal{T}=\inf\{t>0: X_{i}(t)=X_{j}(t) \quad\text{for some}\quad i \neq j; i,j=1,\ldots,d\}$$
	is infinite almost surely.
	\end{dl}
\begin{proof}
It is straightforward to verify that if conditions (B1)-(B5) hold, all the assumptions of  Theorem 2.2 in  \cite{GrMa14} also hold. Thus, the equation \eqref{eqn1.1} has a strong unique solution with no explosion and collision.
\end{proof}
	\subsection{The existence of negative moments in the case that $b_i \equiv 0$}
	In this section we consider the case  that $b_i \equiv 0$ for $i=1, \ldots, d.$

	\begin{bd}\label{bd2}
		Suppose that Assumptions (B2) and (B4) hold.  Then for any $q>0$, there exists a finite positive constant $c = c(q, v, d,  \lambda,T)$ such that $\sup _{0 \leq t \leq T}\mathbb E\left(\|X_{t \wedge \tau}\|^{2 q}\right)\leq c$ for any stopping time $\tau$.
	\end{bd}
	
	\begin{proof}
			It is sufficient to prove the statement for $q \geq 3$ . It follows from \eqref{eqn1.1} and  It\^o formula that 
		\begin{align*}
			d X_{i}^{2}\left(t \right) &=2 X_{i}\left(t \right)\sigma_{i}\left(X_{i}(t)\right)  d B_{i}(t) +2 \sum\limits_{j \neq i} \dfrac{\lambda X_{i}(t)}{X_{i}(t)-X_{j}(t)}  d t + \sigma_{i}^{2}\left(X_{i}(t)\right)  d t.
		\end{align*} 
		This implies that 
		\begin{equation}\notag
			d\left|X\left(t \right)\right|^{2} \\
			=\sum\limits_{i=1}^{d} 2 X_{i}\left(t\right)\sigma_{i}\left(X_{i}(t)\right)d B_{i}(t)+ \lambda d(d-1)  d t +\sum\limits_{i=1}^{d} \sigma_{i}^{2}\left(X_{i}(t)\right)  d t.
		\end{equation}
	Applying It\^o formula again, we get 
		\begin{align*}
			d\left|X\left(t \right)\right|^{2 q} 
			=& 2q \sum\limits_{i=1}^{d}  X_{i}\left(t \right)\sigma_{i}\left(X_{i}(t)\right)\left|X\left(t \right)\right|^{2 q-2}dB_i(t)   + q \lambda d(d-1) \left|X\left(t \right)\right|^{2 q-2}  d t \\
	    &+q \sum\limits_{i=1}^{d} \sigma_{i}^2\left(X_{i}(t)\right)\left|X\left(t \right)\right|^{2 q-2}  d t  +\sum\limits_{i=1}^{d} 2(q-1) q\left|X\left(t\right)\right|^{2 q-4}X_{i}^2\left(t \right)\sigma_{i}^{2}\left(X_{i}(t)\right)  dt.
		\end{align*}
			For each  $n \in \mathbb{N}$, let 
		\begin{equation} \label{def:T_n} 
		T_{n}=\inf \left\{0 \leq t \leq T: \max _{0 \leq i, j \leq d}\left\{\max \left\{|X_i(t)|,\left|X_{i}(t)-X_{j}(t)\right|, \dfrac{1}{\left|X_{i}(t)-X_{j}(t)\right|}\right\}\right\} \geq n  \right\}, 
		\end{equation} 
		 then $T_{n}$ is a  stopping time. For any stopping time $\tau$, we have 
	\begin{align*}
			\mathbb E \left[\|X(t\wedge \tau \wedge T_{n})\|^{2 q}\right]  			 \leq &\mathbb E \left(|X(0)|^{2 q}\right)+ \lambda qd(d-1)  \displaystyle\int\limits_{0}^{t} \mathbb E \left(\left|X\left(s \wedge T_{n}\wedge \tau \right)\right|^{2 q-2}\right) d s \\
			&+q \|\sigma\|^2_\infty d \displaystyle\int\limits_{0}^{t} \mathbb E\left(\left|X\left(s \wedge T_{n}\wedge \tau\right)\right|^{2 q-2}\right) d s \\
			&+2(q-1) q \|\sigma\|^2_\infty \displaystyle\int\limits_{0}^{t} \mathbb E\left(\left|X\left(s \wedge T_{n}\wedge \tau\right)\right|^{2 q-2}\right) d s.
		\end{align*} 
	Thanks to the H\"older  and Gronwall inequalities, there exists a finite positive constant  $c= c(q, v, d,  \lambda,T)$, which is independent of $n$, and that $
	\sup_{0 \leq t \leq T} \mathbb E\left(\|X_{t \wedge \tau \wedge T_n}\|^{2 q}\right)\leq c.$ Let $n$ tend to infinity and apply Fatou's lemma, we conclude the proof.
		\end{proof} 

	\begin{dl}\label{bd1}
	Suppose that Assumptions (B2) and (B4) hold and $2\lambda > \| \sigma\|_\infty^2$. 
	Then there exists a unique strong non-exploding and non-colliding solution of \eqref{eqn1.1} in the time interval $[0,T]$. Moreover, for each positive constant $p \leq  \frac{1}{6}(\frac{2\lambda}{\|\sigma\|_\infty^2}-1)$, there exists a finite  positive constant $C= C(p,v, d, \lambda,T)$ such that
		\begin{equation} \label{eqn2}
			\max _{0 \leq i< j \leq d} \sup_{ t \in [0,T]}  \mathbb E_{\mathbb{P}}\left(\left|X_{i}(t)-X_{j}(t)\right|^{-p}\right)\leq C.
		\end{equation} 
	\end{dl}
	
	\begin{proof}
		The existence and uniqueness of the solution to equation   \eqref{eqn1.1}  is a corollary of Theorem \ref{theorem2.2}. 
		So it is sufficient to prove \eqref{eqn2}.
		Without loss of generality, we suppose that there exists a $d$ - dimensional Brownian motion  $\left(B_{d+1},B_{d+2} \ldots, B_{2 d}\right)$  which  is adapted to the filtration $(\mathcal{F}_t)$ and  independent of $(B_{1},\ldots,B_{d})$. Otherwise, we can extend the original probability space $(\Omega, \mathcal F, \mathbb P)$. Then 
		 $\left(B_{1}, \ldots, B_{2 d}\right)$ is a $2 d$ - dimensional Brownian motion defined on the probability space $\left(\Omega, \mathcal{F} ,\mathbb P\right)$ with the filtration $(\mathcal{F}_{t})_{0 \leq t \leq T}$ satisfying  the usual conditions.
		
		For each positive interger $n$, let $T_n$ be defined as in \eqref{def:T_n}. 
		We consider adaptive processes $F^n$ and $Z^n$ given by   
	\begin{equation}\label{def:F}  
	F^n_{i+d}(t)=B_{i+d}(t)+ 2p \sum\limits_{j \neq i} \int\limits_{0}^{t}  \dfrac{1_{\left\{s \leq T_{n}\right\}} \sigma_{i}\left(X_{i}(s)\right)}{X_{i}(s)-X_{j}(s)} d s, \quad i=1, \ldots, d,
	\end{equation} 
	and 
	\begin{equation*}
	Z^n_{t}= \sum\limits_{i=1}^{d}\int_0^t  \sum\limits_{j \neq i} \dfrac{- 2p \sigma_{i}\left(X_{i}(s)\right) 1_{\left\{s \leq T_{n}\right\}}}{X_{i}(s)-X_{j}(s)} d B_{i+d}(s).
	\end{equation*} 	
	Thanks to definition of $T_n$, the process $\left(\exp \left(Z^n_{t}-\dfrac{\langle Z^n\rangle_{t}}{2}\right), \mathcal{F}_t \right)_{0\leq t \leq T}$ is a martingale. Consider a probability measure $\mathbb Q_{n}$ defined by
		$$\frac{d\mathbb Q_{n}}{d \mathbb P}\Big|_{\mathcal F_t} = \exp \left(Z^n_{t}-\dfrac{\langle Z^n\rangle_{t}}{2}\right).$$
By applying the multi-dimensional Girsanov theorem,  $\left(B_{1}(t), \ldots, B_{d}(t), F^n_{d+1}(t), \ldots, F^n_{2 d}(t)\right)_{0 \leq t \leq T}$ is a $2 d$-dimensional Brownian motion with respect to $\left(\Omega, \mathcal F,  (\mathcal F_t)_{0\leq t \leq T},\mathbb Q_{n}\right)$.
		For all $d\geq m>k \geq 1$ we have:
		$$
		\begin{aligned}
			d\left(X_{m}\left(t \wedge T_{n}\right)\right.&\left.-X_{k}\left(t \wedge T_{n}\right)\right) \\
			&\left.=\sigma_{m}\left(X_{m}(t)\right) 1_{\{t \leq T_{n}\}} d B_{m}(t)-\sigma_{k}\left(X_{k}(t)\right)\right. 1_{\{t \leq T_{n}\}} d B_{k}(t) \\
			&+\sum\limits_{i \neq m, k}\left(\dfrac{\lambda}{X_{m}(t)-X_{i}(t)}-\dfrac{\lambda}{X_{k}(t)-X_{i}(t)}\right) 1_{\left\{t \leq T_{n}\right\}} d t \\
			&+\dfrac{2 \lambda}{X_{m}(t)-X_{k}(t)} 1_{\left\{t \leq T_{n}\right\}} d t. 
		\end{aligned}
		$$
		
	\noindent	By using It\^o formula,
		$$
		\begin{array}{l}
			d \log \left(X_{m}\left(t \wedge T_{n}\right)-X_{k}\left(t \wedge T_{n}\right)\right) \\
			=\dfrac{\sigma_{m}\left(X_{m}(t)\right) 1_{\{t \leq T_{n}\}}}{X_{m}(t)-X_{k}(t)} d B_{m}(t)-\dfrac{\sigma_{k}\left(X_{k}(t)\right) 1_{\{t \leq T_{n}\}}}{X_{m}(t)-X_{k}(t)} d B_{k}(t) \\
			+\sum\limits_{i \neq m, k}\left(\dfrac{\lambda}{\left(X_{m}(t)-X_{i}(t)\right)\left(X_{m}(t)-X_{k}(t)\right)}-\dfrac{\lambda}{\left(X_{k}(t)-X_{i}(t)\right)\left(X_{m}(t)-X_{k}(t)\right)}\right) 1_{\left\{t \leq T_{n}\right\}} d t \\
			+\dfrac{2 \lambda 1_{\{t\leq T_n\}}}{\left(X_{m}(t)-X_{k}(t)\right)^{2}} d t
			-\dfrac{1}{2} \cdot \dfrac{(\sigma_{m}{ }^{2}\left(X_{m}(t)\right)+\sigma_{k}{ }^{2}(X_{k}(t)) 1_{\left\{t \leq T_{n}\right\}}}{\left(X_{m}(t)-X_{k}(t)\right)^{2}} d t.
		\end{array}
		$$
		
	\noindent	We will use  the following identity, 
		\begin{equation} \label{eqn2.4}
			\sum\limits_{d \geq m>k \geq 1} \sum\limits_{i \neq m, k}\left\{\dfrac{1}{\left(x_{m}-x_{i}\right)\left(x_{m}-x_{k}\right)}-\dfrac{1}{\left(x_{k}-x_{i}\right)\left(x_{m}-x_{k}\right)}\right\}=0,
		\end{equation}
		which holds for all $x \in \Delta_{d}$. 
We obtain that 
		\begin{align}  
			&d \sum\limits_{d \geq m>k \geq 1} 2p . \log \left(X_{m}\left(t \wedge T_{n}\right)-X_{k}\left(t \wedge T_{n}\right)\right) \notag 
			\\
			=&\sum\limits_{1 \leq m \leq d} \sum\limits_{k \neq m} \dfrac{2p \sigma_{m}\left(X_{m}(t)\right) 1_{\left\{t \leq T_{n}\}\right.}}{X_{m}(t)-X_{k}(t)} d B_{m}(t) \notag  \\
			 &+\sum\limits_{d \geq m>k \geq 1} \dfrac{2p\left(2 \lambda-\dfrac{1}{2}(\sigma_{m}^{2}\left(X_{m}(t)\right)+\sigma_{k}^{2}\left(X_{k}(t)\right))\right) 1_{\left\{t \leq T_{n}\right\}}}{\left(X_{m}(t)-X_{k}(t)\right)^{2}} d t. \label{eqn2.5}
		\end{align}
	Let  $H^n_{t}, I^n_{t}, J^n_{t}$ be stochastic processes defined by
		\begin{align*}
		    & H^n_0 = I^n_0 = J^n_0 = 0,\\ 
			&d H^n_{t}= \sum\limits_{d \geq m>k \geq 1} \dfrac{2p\left(2 \lambda-\dfrac{1}{2}(\sigma_{m}^{2}\left(X_{m}(t)\right)+\sigma_{k}^{2}\left(X_{k}(t)\right))\right) 1_{\left\{t \leq T_{n}\right\}}}{\left(X_{m}(t)-X_{k}(t)\right)^{2}} d t,\\
			&d I^n_{t}= \sum\limits_{1 \leq m \leq d} \sum\limits_{k \neq m} \dfrac{2p \sigma_{m}\left(X_{m}(t)\right) 1_{\{t \leq T_{n}\}}}{(X_{m}(t)-X_{k}(t))} d B_{m}(t), \\
			&d J^n_{t}=\sum\limits_{1 \leq m \leq d} \sum\limits_{k \neq m} \dfrac{2p \sigma_{m}\left(X_{m}(t)\right) 1_{\{t \leq T_{n}\}}}{X_{m}(t)-X_{k}(t)} d F^n_{m+d}(t).
		\end{align*}

		It follows from \eqref{def:F} that  $\frac{d \mathbb P}{d \mathbb Q_{n}} = \exp \left(J^n_{t}-\dfrac{\langle J^n\rangle_{t}}{2}\right)$. Thus, 
	\begin{align}
			S_n &:= \mathbb E_{\mathbb P}\left(\prod_{d \geq m>k \geq 1}\left(X_{m}\left(t \wedge T_{n}\right)-X_{k}\left(t \wedge T_{n}\right)\right)^{-2p}\right) \notag\\
			&=\mathbb E_{\mathbb Q_n}\left(\prod_{d \geq m>k \geq 1}\left(X_{m}\left(t \wedge T_{n}\right)-X_{k}\left(t \wedge T_{n}\right)\right)^{-2p} \exp \left(J^n_{t}-\dfrac{1}{2}\langle J^n\rangle_{t}\right)\right) \notag.
	\end{align}
	Using the decomposition 
	$$ J^n_{t}-\dfrac{1}{2}\langle J^n\rangle_{t} = (I^n_t + H^n_t) + (J^n_t - I^n_t -  \langle J^n - I^n \rangle_t ) + \big( \langle J^n - I^n \rangle_t - H^n_t-\frac 12 \langle J \rangle _t\big),$$
	and the estimate 
	\begin{equation} \label{cmsau}
	\langle J^n - I^n \rangle_t - H^n_t-\frac 12 \langle J^n \rangle _t \leq 0,
	\end{equation} 
	which will be proved later, we get 
	\begin{align}
S_n &\leq \mathbb E_{\mathbb Q_n} \left[\Big(\prod _ { d \geq m > k \geq 1 } ( X _ { m } ( t \wedge T _ { n } ) - X _ { k } ( t \wedge T _ { n } ) ) ^ { - 2p } \Big)\exp \left(I^n_{t}\right.\right.+H^n_{t}) \notag\\
			&\indent \indent \cdot \exp \left(J^n_{t}-I^n_{t}-\langle J^n-I^n\rangle_{t}\right) \bigg].\notag
	\end{align}
	Next, using the Cauchy-Schward inequality and equation \eqref{eqn2.5}, we get, 
		\begin{align}
	S_n^2   	&\leq \mathbb E_{\mathbb Q_n}\left[\bigg(\prod_{d \geq m>k \geq 1}(X_{m}(t \wedge T_{n})-X_{k}(t \wedge T_{n}))^{-4 p}\bigg) \exp \Big(2\left(I^n_{t}+H^n_{t}\right)\Big)\right] \cdot \notag\\		
	&\indent \mathbb E_{\mathbb Q_n}\Big(\exp \left(2\left(J^n_{t}-I^n_{t}\right)-2\langle J^n-I^n \rangle_{t}\right)\Big) \notag\\
			&= a_{p}^2 \cdot \mathbb E_{\mathbb Q_n}\Big(\exp \left(2\left(J^n_{t}-I^n_{t}\right)-2\langle J^n-I^n \rangle_{t}\right)\Big),
			\notag 
        \end{align}
where  $a_{p}=\exp \left[ \sum\limits_{d \geq m>k \geq 1} 2p \cdot \log \left(X_{m}(0)-X_{k}(0)\right)\right]$.

	Note that  $\exp \left(2\left(J^n_{t}-I^n_{t}\right)-2\langle J^n - I^n \rangle_{t}\right)$ is a martingale with respect to $\mathbb Q_n$, which implies $$\mathbb E_{\mathbb Q_n}\left(\exp \left(2\left(J^n_{t}-I^n_{t}\right)-2\langle J^n-I^n\rangle_{t}\right)\right)=1.$$
			Therefore, 
		\begin{equation}\label{danhgiatich}
		\mathbb E_{\mathbb P}\left(\prod_{d \geq m>k \geq 1}\left(X_{m}\left(t \wedge T_{n}\right)-X_{k}\left(t \wedge T_{n}\right)\right)^{-p}\right) \leq a_{p}.
		\end{equation}
	In the following, we will show that  $\langle J^n-I^n\rangle_{t} \leq H^n_{t}+\frac 12\langle J^n \rangle _t.$
		Indeed, we have
			\begin{align*}
		\langle J^n -I^n \rangle_{t}-\frac 12 \langle J^n \rangle_t 	 = &  \frac 32 \int_0^t \sum\limits_{m=1}^{d}\left(\sum\limits_{k \neq m} \dfrac{2p \sigma_{m}\left(X_{m}(s)\right) 1_{\{s \leq T_{n}\}}}{X_{m}(s)-X_{k}(s)}\right)^{2} ds \\
		= &  6p^2 \int_0^t \sum\limits_{m=1}^{d} \sigma_{m}\left(X_{m}(s)\right)^2 1_{\{s \leq T_{n}\}} \left(\sum\limits_{k \neq m} \dfrac{1 }{X_{m}(s)-X_{k}(s)}\right)^{2} ds \\
		\leq & 6 p^{2} \|\sigma\|^{2}_{\infty} \int_0^t 1_{\{s \leq T_{n}\}} \sum\limits_{m=1}^{d} \sum\limits_{k \neq m}\dfrac{1}{\left(X_{m}(s)-X_{k}(s)\right)^{2}} ds  \\
			 & + 6 p^{2} \|\sigma\|^{2}_{\infty} \int_0^t 1_{\{s \leq T_{n}\}} \sum\limits_{m=1}^{d} \sum\limits_{h \neq k \neq m} \dfrac{1}{\left(X_{m}(s)-X_{k}(s)\right)\left(X_{m}(s)-X_{h}(s)\right)}ds  \\
			=&\int_0^t 12p^{2} \|\sigma\|^{2}_{\infty} 1_{\{s \leq T_{n}\}} \sum\limits_{m=1}^{d} \sum\limits_{m>k }\dfrac{1}{\left(X_{m}(s)-X_{k}(s)\right)^{2}}ds,
		\end{align*}
		where  the last equation is a consequence of identity \eqref{eqn2.4}.
		Since $(6p+1)\|\sigma\|_\infty^2 \leq 2\lambda$, we get  $2p(2\lambda - \|\sigma\|^{2}_{\infty}) \geq 12p^2\|\sigma\|_\infty^2 $. Hence, 
		\begin{align*}
			 H^n_{t} & \geq \displaystyle \int \limits _{0}^{t}\sum\limits_{d \geq m>k \geq 1} \dfrac{2p\left(2 \lambda-\|\sigma\|_{\infty}^{2}\right) 1_{\left\{s \leq T_{n}\right\}}}{\left(X_{m}(s)-X_{k}(s)\right)^{2}} d s
			 & \geq  \displaystyle \int \limits _{0}^{t}\sum\limits_{d \geq m>k \geq 1} \dfrac{12p^2 \|\sigma\|_\infty^2 1_{\left\{s \leq T_{n}\right\}}}{\left(X_{m}(s)-X_{k}(s)\right)^{2}} d s \geq 
		 \langle J^n -I^n \rangle_{t}-\frac 12 \langle J^n \rangle_t. \notag 
		 \end{align*} 
	For all $d \geq i>j\geq 1$ we have
		\begin{align*}
			&\mathbb E_{\mathbb P}\left(\left(X_{i}\left(t \wedge T_{n}\right)-X_{j}\left(t \wedge T_{n}\right)\right)^{-p}\right)\\
			&\leq \mathbb E_{\mathbb P}\left(\prod_{d \geq m>k \geq 1} \left(X_{m}\left(t \wedge T_{n}\right)-X_{k}\left(t \wedge T_{n}\right)\right)^{-2p}\right) + \mathbb E_{\mathbb P}\left(\prod_{(m, p) \neq(i, j)}\left(X_{m}\left(t \wedge T_{n}\right)-X_{k}\left(t \wedge T_{n}\right)\right)^{2p}\right).
		\end{align*}
		
		From \eqref{danhgiatich} and the result of Lemma \ref{bd2}, there exists a constant $C_{p}>0$ which does not depend on $n$, such that $$\mathbb E_{\mathbb P}\left(\left(X_{i}\left(t \wedge T_{n}\right)-X_{j}\left(t \wedge T_{n}\right)\right)^{-p}\right) \leq C_{p},$$  for all $d \geq i>j \geq 1, t \in[0, T].$
		Let $n$ tend to infinity and apply Fatou's lemma, we obtain
		$$\sup_{0 \leq t \leq T} \max _{0 \leq i, j \leq d} \mathbb E_{\mathbb P}\left(\left|X_{i}(t)-X_{j}(t)\right|^{-p}\right)\leq C_{p}.$$
		We conclude the proof.
		
	\end{proof} 
		\subsection{The existence of negative moments for non-zero function $b$}
	Throughout this section, we suppose that  $\sigma$ is uniformly elliptic.
	\begin{GT}[C] 
		$L^{2}:= \max \limits _{1 \leq i \leq d} \sup \limits _{x \in \mathbb{R}} \dfrac{1}{\sigma_{i}^{2}(x)}< \infty$.
	\end{GT}

	\begin{dl}\label{bd1.2}
	 Suppose that the Assumptions (B1)-(B5) and (C) hold. In addition, suppose that $2\lambda > \|\sigma\|_\infty^2$. 
	Then the system \eqref{eqn1.1} has a strong and path-wise unique solution on $[0,T]$. For any positive constant $p < \frac 16 ( \frac{2\lambda}{\| \sigma\|_\infty^2} -1)$, there exists a finite positive constant $C = C(p,v, d, \lambda, \|b\|_\infty, L, T)$ such that
	\begin{equation} \label{dg01} 
	\sup _{0 \leq t \leq T}	\max _{0 \leq i< j \leq d} \mathbb E_{\mathbb{P}}  \left(\left|X_{i}(t)-X_{j}(t)\right|^{-p}\right) \leq C.
	\end{equation} 
	Moreover, for any $q>0$,  there exists a finite positive constant $c= c(q,v, d, \lambda, \|b\|_\infty, L, T)$, such that 
	\begin{equation} \label{dg02}
	\sup_{t \in [0,T]} \mathbb E_{\mathbb P}\left(\|X_{t }\|^{2 q}\right)\leq c.
	\end{equation} 
	\end{dl}

\begin{proof}
The existence and uniqueness of the solution to \eqref{eqn1.1} is deduced from  Theorem \ref{theorem2.2}. Let us denote    $W_{i}(t)=B_{i}(t)+\displaystyle\int\limits_{0}^{t} \dfrac{b_{i}\left(X_{i}\right)}{\sigma_{i}\left(X_{i}\right)} d s$, and $ G_{t}=\int \limits _{0}^{t}\sum\limits_{i=1}^{d}  \dfrac{-b_{i}\left(X_{i}\right)}{\sigma_{i}\left(X_{i}\right)} d B_{i}$. Since $b$ is bounded and $\sigma$ is uniformly elliptic,  the process $\left(\exp \left(G_{t}-\dfrac{\langle G\rangle_{t}}{2}\right), \mathcal{F}_t \right)_{0\leq t \leq T}$ is a martingale. We  consider the probability measure  $\mathbb Q$ defined by
	$$\frac{d \mathbb Q}{d\mathbb P} \Big| _{F_{t}} = \exp \left(G_{t}-\dfrac{\langle G\rangle_{t}}{2}\right), \quad 0 \leq t \leq T.$$
Thanks to  the assumptions (B3) and (C), it follows from the Girsanov theorem that  $\left(W_{1}(s), \ldots, W_{d}(s)\right)_{s \in [0,T]}$ is a  $d$-dimensional Brownian motion with respect to $\mathbb Q$. 
Rewrite equation \eqref{eqn1.1} under probability measure $\mathbb Q$, one has 
	$$
	\left\{\begin{array}{c}
		d X_{i}=\sigma_{i}\left(X_{i}\right) d W_{i}+\sum\limits_{j \neq i} \dfrac{\lambda}{X_{i}-X_{j}} d t, \quad i=1, \ldots, d, \\
		\left(X_{1}(0), \ldots, X_{d}(0)\right)=v \in \Delta_{d}.
	\end{array}\right.
	$$
	
\noindent	Denote by $U_{t}= \int _{0}^{t}\sum \limits_{i=1}^{d}\dfrac{b_{i}(X_{i})}{\sigma_{i}(X_{i})}dW_{i}$. 
We have 
$\langle U \rangle_{t}= \int _{0}^{t}\sum \limits_{i=1}^{d}\dfrac{b^2_{i}(X_{i})}{\sigma^2_{i}(X_{i})}ds \leq db^2L^2T$.
According to the Girsanov theorem, for all $d\geq i >j \geq 1:$
\begin{align*} 
		\mathbb E_{\mathbb P}\left(\left|X_{i}(t)-X_{j}(t)\right|^{-p}\right) &=\mathbb E_{\mathbb Q}\left(\left|X_{i}(t)-X_{j}(t)\right|^{-p} \exp \left(U_{t}-\dfrac{\langle U\rangle_{t}}{2}\right)\right) \\
		&\leq\left[\mathbb E_{\mathbb Q}\left(\left|X_{i}(t)-X_{j}(t)\right|^{-r}\right)\right]^{\frac{p}{r}} \cdot\left[\mathbb E_{\mathbb Q}\left(\exp \left(\frac{r}{r-p}\left(U_{t}-\frac{\langle U\rangle_{t}}{2}\right)\right)\right)\right]^{\frac{r-p}{r}}, 
\end{align*} 
where $ r = \frac{2 \lambda - \|\sigma\|^2_\infty}{6\|\sigma\|^2_\infty} > p$. 
\noindent	The first term in the last expression is bounded because of  Theorem \ref{bd1}. We estimate the second term
\begin{align*}
    		&\mathbb E_{\mathbb Q}\left(\exp \left(\dfrac{r}{r-p}\left(U_{t}-\dfrac{\langle U\rangle_{t}}{2}\right)\right)\right) \\
		=&\mathbb E_{\mathbb Q}\left(\exp \left(\dfrac{r}{r-p} U_{t}-\dfrac{r^{2}\langle U\rangle_{t}}{2(r-p)^{2}}\right) \cdot \exp \left(\left(\dfrac{r^{2}}{2(r-p)^{2}}-\dfrac{1}{2}\right)\langle U\rangle_{t}\right)\right)\\
		\leq &  \exp \left(\left(\dfrac{r^{2}}{2(r-p)^{2}}-\dfrac{1}{2}\right) db^2L^2T \right)
		\mathbb E_{\mathbb Q}\left(\exp \left(\dfrac{r}{r-p} U_{t}-\dfrac{r^{2}\langle U\rangle_{t}}{2(r-p)^{2}}\right) \right).
\end{align*}
	
\noindent	Because $b_{i}$ and $\sigma_{i}$ are bounded,  $\langle U\rangle_{t}$ is bounded and $\exp \left(\dfrac{r}{r-p} U_{t}-\dfrac{r^{2}\langle U\rangle_{t}}{2(r-p)^{2}}\right)$
	is a martingale under probability measure  $\mathbb Q$. Therefore, $\mathbb E_{\mathbb Q}\left(\exp \left(\dfrac{r}{r-p} U_{t}-\dfrac{r^{2}\langle U\rangle_{t}}{2(r-p)^{2}}\right) \right)= 1$,  which
implies \eqref{dg01}. 
%
	
	Next, we prove \eqref{dg02}. For any $q >0$, 
	it follows from the Girsanov theorem and H\"older's inequality that 
	\begin{align*} 
	\mathbb E_{\mathbb P}\left(\|X_{t }\|^{2 q}\right) & =\mathbb E_{\mathbb Q}\left(\|X_{t }\|^{2 q} \exp \left(U_{t}-\dfrac{\langle U\rangle_{t}}{2}\right)\right) \\
	&\leq\left[\mathbb E_{\mathbb Q}\left(\|X_{t }\|^{\frac{2qr}{p}}\right)\right]^{\frac{p}{r}} \cdot\left[\mathbb E_{\mathbb Q}\left(\exp \left(\dfrac{r}{r-p}\left(U_{t}-\dfrac{\langle U\rangle_{t}}{2}\right)\right)\right)\right]^{\frac{r-p}{r}}\\
	&= \left[\mathbb E_{\mathbb Q}\left(\|X_{t }\|^{\frac{2qr}{p}}\right)\right]^{\frac{p}{r}}  \exp \left(\frac{r-p}{r} \left(\dfrac{r^{2}}{2(r-p)^{2}}-\dfrac{1}{2}\right) db^2L^2T \right)
	\end{align*}
	Thanks to  Lemma \ref{bd2}, the last term is bounded.  We obtain \eqref{dg02}. 
\end{proof}

	\subsection{Applications}
	In this section, we will apply the existence of the negative moments to study the convergence rate of some numerical approximation schemes for equation \eqref{eqn1.1}. The key point is to   verify Hypothesis 2.7 in \cite{NT17} and Hypothesis $H_{\hat{p}}$ in \cite {LN}. 

	\subsubsection{Semi-implicit Euler-Maruyma scheme}

We recall the semi-implicit Euler-Maruyama  scheme for \eqref{eqn1.1}, which was introduced in \cite{NT17}. For each positive interger $n$, let 
	$X^{(n)}(0)=X(0).$ 
	For each $0\leq k \leq n$, set   $t^{(n)}_k = kT/n$, and 
	$X^{(n)}(t_{k+1}^{(n)})$ is the unique solution in $\Delta_{d}$ of the following system of  equations
	\begin{align*}
		X_{i}^{(n)}\left(t_{k+1}^{(n)}\right) =& X_{i}^{(n)}\left(t_{k}^{(n)}\right) 
		+\left\{ \sum_{j \neq i} \frac{\lambda}{X_{i}^{(n)}\left(t_{k+1}^{(n)}\right)-X_{j}^{(n)}\left(t_{k+1}^{(n)}\right)}+b_{i}\left(X_{i}^{(n)}\left(t_{k}^{(n)}\right)\right) \right\}\dfrac{T}{n}\\
	\indent	&+\sigma_{i}\left(X_{i}^{(n)}\left(t_{k}^{(n)}\right)\right)\left\{B_{i}\left(t_{k+1}^{(n)}\right)-B_{i}\left(t_{k}^{(n)}\right)\right\}, \quad 1\leq i \leq d.
	\end{align*}
	 

\begin{dl}\label{bd7}
Suppose that $b=(b_i)_{1 \leq i \leq d}$ and $\sigma = (\sigma_i)_{1 \leq i \leq d}$ are globally Lipschitz continuous on $\mathbb{R}$. In addition, let assumptions  (B3), (B5), (C) hold, and
$\lambda > \frac{37}{2} \| \sigma \|_\infty.$
Let $X$ be the unique strong solution of equation \eqref{eqn1.1}. 
Then the semi-implicit Euler-Maruyama scheme $X^{(n)}$ converges at rates of order $\dfrac{1}{2}$ in $L^{2}$, i.e., there exists a positive constant $C= C(v,d,\lambda,b, \sigma, T)$ such that
	\begin{equation}\label{rate1}
	\sup _{1 \leq k \leq n} \mathbb  E \left[\left|X\left(t_{k}^{(n)}\right)-X^{(n)}\left(t_{k}^{(n)}\right)\right|^{2}\right] \leq \dfrac{C}{n}.
	\end{equation} 
\end{dl}
\begin{proof}
	Thanks to Theorem \ref{bd1.2}, using the argument in the of proof of Theorem 3.6 in \cite{NT17}, we can show that  
	\begin{equation}\notag
	\max \limits_ {1 \leq i \leq d} \mathbb E_{\mathbb P}\left[\left|X_{i}(t)-X_{i}(s)\right|^{6}\right] \leq C (t-s)^{3},\forall 0 \leq s \leq t \leq T,
	\end{equation}
	for some finite positive constant	$C$ Therefore, the Hypothesis 2.7 in \cite{NT17} is satisfied with $\hat{p} = 6$. Again, by following the argument in the proof of Theorem 2.10 in \cite{NT17}, we can obtain the desired result. 
\end{proof} 
For the case $b_i \equiv 0$ for all $i=1,\ldots, d$ we can show the rate of convergence of the semi-implicit Euler-Maruyama scheme without assumming condition (C). In particular, we have the following result. 
\begin{dl}\label{bd7b}
	Suppose that $b_i \equiv 0$ for all $i=1,\ldots, d$, $\sigma = (\sigma_i)_{1 \leq i \leq d}$ are globally Lipschitz continuous on $\mathbb{R}$, and
	$\lambda > \frac{37}{2} \| \sigma \|_\infty.$
	Let $X$ be the unique strong solution of equation \eqref{eqn1.1}. 
	Then the conclusion \eqref{rate1} holds. 
\end{dl}

\subsubsection{Semi-implicit Milstein scheme}
We recall the semi-implicit Milstein scheme for \eqref{eqn1.1}, which was introduced in \cite{LN}. The scheme is defined as follow. Suppose that Assumption (C3) holds. For each integer $n \geq 1$ and $1 \leq k \leq n$, set $t^{(n)}_{k} = kT/n$, $X^{(n)}(0):=X(0)$. Moreover, for each $k=0,\ldots, n-1$ and $t \in [t^{(n)}_{k},t^{(n)}_{k+1}]$, $X^{(n)}(t)=(X^{(n)}_{i}(t))_{1 \leq i \leq d}$ is the unique solution in $\Delta _{d}$ of the following system of equations
	\begin{align*}
		X_{i}^{(n)}(t)
		&=X_{i}^{(n)}\left(t_{k}^{(n)}\right) 
		+\left\{ \sum_{j \neq i} \frac{\lambda}{X_{i}^{(n)}(t)-X_{j}^{(n)}(t)}+b_{i}\left(X_{i}^{(n)}\left(t_{k}^{(n)}\right)\right) \right\}\big(t-t^{(n)}_{k}\big)\\
    	&+\sigma_{i}\left(X_{i}^{(n)}\left(t_{k}^{(n)}\right)\right)\left\{B_{i}(t)-B_{i}\left(t_{k}^{(n)}\right)\right\}\\
    	&+\frac{1}{2}\sigma_{i}\left(X_{i}^{(n)}\left(t_{k}^{(n)}\right)\right)\sigma'_{i}\left(X_{i}^{(n)}\left(t_{k}^{(n)}\right)\right)\bigg\{\Big(B_{i}(t)-B_{i}(t^{(n)}_{k})\Big)^{2}-\big(t-t^{(n)}_{k}\big)\bigg\}, \quad i=1,\ldots,d.
	\end{align*}
	
Let $C^{2}_b{(\mathbb R)}$ denote the set of all fuctions $f:\mathbb R \to \mathbb R$ such that $f,f',f''$ are bounded.
The following theorem is straightaway from  Theorem \ref{bd1.2} and  Theorem 2.1 in \cite{LN}.
\begin{dl}\label{milstein convergence}
	Assume that $b_i, \sigma_i \in C^{2}_b{(\mathbb R})$ for all $i=1,\ldots,d$. In addition, suppose that all assumptions of Theorem \ref{bd7} also hold.  Then the semi-implicit Milstein scheme converges at rates of order $1$ in $L^{2}$, i.e., there exists a finite positive constant $C= C(v,d,\lambda,b, \sigma, T)$ such that
		$$ \sup _{k=1, \ldots, n} \mathbb E \left[\left|X\left(t_{k}^{(n)}\right)-X^{(n)}\left(t_{k}^{(n)}\right)\right|^{2}\right] \leq \dfrac{C}{n^{2}}.$$
\end{dl}

		

	\newpage


\begin{thebibliography}{99}
		
		\bibitem{AGZ10}
		{Anderson, G.~W.}, {Guionnet, A.} and {Zeitouni, O.}
		{\it An Introduction to Random Matrices.}
		volume 118 of {\em Cambridge Studies in Advanced Mathematics}.
		Cambridge University Press, Cambridge,
		(2010).
		
		%
		%
		%
		%
		%
		%
		%
		\bibitem{Bru90}
		{Bru, M.~F.}
		{Wishart processes.}
		{\it J. Theoret. Probab.}
		{\bf 4}(4)
		725--751
		(1991).
		
		\bibitem{CeLe97}
		{C\'epa, E.} and {L\'epingle, D.}
		{Diffusing particles with electrostatic repulsion.}
		{\it Probab. Theory Related Fields}
		{\bf 107}(4)
		429--449
		(1997).
		
		\bibitem{CeLe01}
		{C\'epa, E.} and {L\'epingle, D.}
		{Brownian particles with electrostatic repulsion on the circle:Dyson's model for unitary random matrices revisited.}
		{\it ESAIM Probab. Statist.}
		{\bf 5}
		203--224
		(2001).
		
		%
		%
		
		%
		%
		%
		%
		%
		%
		%
		%
		%
		%
		
		\bibitem{Dyson}
		{Dyson, F.~J.}
		{\it A Brownian-motion model for the eigenvalues of a random matrix.}
		{\it J. Math. Phys.}
		{\bf 3}(6)
		1191--1198
		(1962).
		
		%
		%
		%
		
		\bibitem{GrMa14}
		{Graczyk, P.} and {Ma\l ecki, J.}
		{Strong solutions of non-colliding particle systems.}
		{\it Electron. J. Probab.}
		{\bf 19}(119)
		1--21
		(2014).
		
		%
		%
		%
		%
		
		%
		
		%
		
		
		%
		%
		
		\bibitem{Ka15}
		{Katori, M.}
		{\it Bessel Processes, {S}chramm-{L}oewner Evolution, and the {D}yson model.}
		Springer
		(2015).
		
		\bibitem{KaTa04}
		{Katori, M.} and {Tanemura, H.}
		{Symmetry of matrix-valued stochastic processes and noncolliding diffusion particle systems.}
		{\it J. Math. Phys.}
		{\bf 45}(8)
		3058--3085
		(2004).
		
		\bibitem{KaTa11}
		{Katori, M.} and {Tanemura, H.}
		{Noncolliding squared Bessel processes.}
		{\it J. Stat. Phys.}
		{\bf 142}(3)
		592--615
		(2011).
		
		%
		
		%
		
		\bibitem{LN}
		Luong, D-T., and Ngo, H-L., Semi-implicit Milstein approximation scheme for non-colliding particle systems. \emph{Calcolo} 56(3), 1--18 (2019).
		
		\bibitem{LiMe}
		{Li, X.~H.} and {Menon, G.}
		{ Numerical solution of Dyson Brownian motion and a sampling scheme for invariant matrix ensembles.}
		{\it J. Stat. Phys.}
		{\bf 153}(5)
		801–-812
		(2013).
		
		
		%
		%
		
		
		
		\bibitem{NT17}
		{Ngo, H.-L.} and {Taguchi, D.}
		{Semi-implicit Euler--Maruyama approximation for non-colliding particle systems.}
		{\it Ann. Appl. Probab.}
		{\bf 30}(2)
		673--705
		(2020).
		
		%
		%
		%
		%
		
		%
		
		%
		%
		
		%
		
		%
		
	\end{thebibliography}
\end{document}